\documentclass[11pt,a4paper]{amsart}
\usepackage{amscd,amssymb,amsopn,amsmath,amsthm,graphics,amsfonts,enumerate,verbatim,calc}
\usepackage[dvips]{graphicx}

\usepackage{color}

\usepackage{amssymb,amsmath}

\headsep=1cm \numberwithin{equation}{section}
\newtheorem{theorem}{Theorem}[section]
\newtheorem{lemma}[theorem]{Lemma}
\newtheorem{proposition}[theorem]{Proposition}

\newtheorem{maintheorem}{Main Theorem}

\theoremstyle{definition}

\theoremstyle{remark}
\newtheorem{remark}[theorem]{Remark}

\newtheorem{acknowledgement}{Acknowledgement}

\newcommand{\Ass}{\operatorname{Ass}}

\newcommand{\Spec}{\operatorname{Spec}}

\newcommand{\Ht}{\operatorname{ht}}

\newcommand{\depth}{\operatorname{depth}}

\newcommand{\Proj}{\operatorname{Proj}}
\newcommand{\Cl}{\operatorname{Cl}}

\newcommand{\fm}{\frak{m}}
\newcommand{\fp}{\frak{p}}

\begin{document}
\title[Hyperplane sections of non-standard graded rings]
{Hyperplane sections of non-standard graded rings}

\author[K. Shimomoto]{Kazuma Shimomoto}
\address{Department of Mathematics, Tokyo Institute of Technology, 2-12-1 Ookayama, Meguro, Tokyo 152-8551, Japan}
\email{shimomotokazuma@gmail.com}




\begin{abstract}
The purpose of this paper is to explain a method on the generalization of the Bertini-type theorem on standard graded rings to the non-standard graded case of certain type.
\end{abstract}

\maketitle 

\tableofcontents

\section{Introduction}

In a previous paper \cite{HS18}, we studied the Bertini theorem for normal hyperplane sections in mixed characteristic. In this article, we prove that a Noetherian graded normal domain with a unique graded maximal ideal defined over a field of characteristic zero has sufficiently many normal specializations under a mild condition. More precisely, we have the following theorem, which generalizes the classical Bertini theorem due to Flenner \cite{Fl77} to the case of non-standard graded rings.

\begin{maintheorem}
\label{maintheoremA}
Let $R=\bigoplus_{n \ge 0} R_n$ be a Noetherian graded normal domain with the unique graded maximal ideal $\fm=\bigoplus_{n>0} R_n$. Assume that $R_0=K$ is an algebraically closed field, the depth of $R_{\fm}$ is at least $3$, and let $x_1,\ldots,x_n \in R_1$ be a set of homogeneous elements in degree $1$. Set
$$
\mathbf{x}_{\alpha}:=\sum_{i=1}^n \alpha_i x_i,
$$
where $\alpha=(\alpha_1,\ldots,\alpha_n) \in K^n$. Then the following assertions hold:
\begin{enumerate}
\item
Assume that $R=K[R_1]$, the elements $x_1,\ldots,x_n$ form a $K$-basis of $R_1$, and the characteristic of $K$ is arbitrary. Then there exists a Zariski-dense open subset $\mathcal{U} \subset K^n$ such that $R/\mathbf{x}_{\alpha}R$ is a graded normal domain for every $\alpha \in \mathcal{U}$.

\item
Assume that the composite map $K[x_1,\ldots,x_n] \hookrightarrow K[R_1] \hookrightarrow R$ is module-finite, $x_1,\ldots,x_n$ form a homogeneous system of parameters for $R$, and the characteristic of $K$ is zero. Then there exists a Zariski-dense open subset $\mathcal{U} \subset K^n$ such that $R/\mathbf{x}_{\alpha}R$ is a graded normal domain for every $\alpha \in \mathcal{U}$.
\end{enumerate}
\end{maintheorem}

In this paper, we offer two different approaches to the proof of the above theorem. The first proof is based on Flenner's local Bertini theorem \cite{Fl77} that applies to $(1)$, while the second one applies to both $(1)$ and $(2)$ of the theorem, based on the ideas of the second Bertini theorem by Cumino, Greco, and Manaresi \cite{CGM86} using projective geometry. The reason for the inclusion of both proofs is that the first one is of algebraic nature and the second one is of algebro-geometric nature. Just as in the condition $(2)$ of Main Theorem \ref{maintheoremA}, a graded ring $R=\bigoplus_{n \ge 0} R_n$ with $R_0=K$ a field is called \textit{semi-standard} if $R$ is module-finite over $K[R_1]$. Graded rings of this type appear in the literature and often arise in the context of Ehrhart ring associated to a lattice polytope (see \cite{HiYan18} and \cite{KY20} for some recent developments).

As an application, we prove a Grothendieck-Lefschetz style result for graded normal domains that are not necessarily assumed to be standard graded.

\begin{maintheorem}
\label{ClassGroup}
Let $R=\bigoplus_{n \ge 0} R_n$ be a Noetherian graded normal domain defined over an algebraically closed field $R_0=K$ of characteristic 0 with the unique graded maximal ideal $\fm=\bigoplus_{n>0}R_n$ and $\dim R \ge 4$. Assume that $x_1,\ldots,x_n \in R_1$ satisfy the condition as in Main Theorem \ref{maintheoremA}(1) or (2), and the depth of $R_\fm$ is at least $3$.

Then there exists a Zariski-dense open subset $\mathcal{U} \subset K^n$ such that $R/\mathbf{x}_{\alpha}R$ is a graded normal domain for every $\alpha \in \mathcal{U}$ and the restriction map between divisor class groups
$$
\Cl(R) \to \Cl(R/\mathbf{x}_{\alpha}R)
$$
is injective.
\end{maintheorem}

The next theorem refines the classical result on normal graded rings over a field as a section ring of some ample $\mathbb{Q}$-divisor in the case of semi-standard graded rings (see \cite{De79}).

\begin{maintheorem}
\label{Non-StandardDemazure}
Let $R=\bigoplus_{n \ge 0} R_n$ be a Noetherian graded normal domain such that $R_0=K$ is a field and $\dim R \ge 2$. Suppose that the natural map $K[R_1] \hookrightarrow R$ is module-finite. Then there exists an ample Cartier divisor $D$ on $X=\Proj(R)$ such that
$$
R \cong \bigoplus_{n \ge 0} H^0(X,\mathcal{O}_X(nD))
$$
as graded rings.
\end{maintheorem}

\section{Bertini theorem for graded rings}

We begin with the following lemma.

\begin{lemma}
\label{PID}
Let $A$ be a principal ideal domain with its field of fractions $K$. Let $A \subset B \subset K$ be a subring. Then $B$ is the localization of $A$ with respect to some multiplicative subset of $A \setminus \{0\}$.
\end{lemma}

\begin{proof}
Let us choose nonzero elements $a, b \in A$. Then since $A$ is a principal ideal domain, there exists $d \in A$ such that $(a,b)=(d)$. Then we have $d|a$ and $d|a$, which implies that $d$ is the greatest common divisor of $a$ and $b$.

Now consider $ab^{-1} \in B$ for $a,b \in A$, where we assume that $a$ and $b$ are relatively prime. Then by the above fact, we find that $sa+tb=1$ for $s,t \in A$. Then
$$
sab^{-1}+t=b^{-1}
$$
from which we have $b^{-1} \in B$. Now let us write $B=A[\Sigma']$ for some subset $\Sigma' \subset K$. Let $\Sigma$ be the multiplicative subset of $A \setminus \{0\}$ which is generated by all elements $b \in A$ for which $ab^{-1} \in \Sigma'$ for some $a \in A$. We conclude that $A[\Sigma'] \subset A[\{b^{-1}\}_{b \in \Sigma}] \subset A[\Sigma']$ and thus, $B=A[\{b^{-1}\}_{b \in \Sigma}]$. This shows that $B$ is the localization of $A$ with respect to $\Sigma$, as required.
\end{proof}

We need the following result of Matijevic-Roberts type. As we are unable to find a reference for it, we give a proof.

\begin{proposition}
\label{localization}
Assume that $R=\bigoplus_{n \in \mathbb{Z}} R_n$ is a $\mathbb{Z}$-graded Noetherian ring. Then $R$ is a normal ring if and only if $R_{\fm}$ is an integrally closed domain for any graded maximal ideal $\fm$ of $R$.
\end{proposition}

\begin{proof}
Let $\fp$ be a prime ideal of $R$. Let $\fp^*$ denote the ideal generated by homogeneous elements contained in $\fp$. Then $\fp^*$ is a prime ideal by  \cite[Lemma 1.5.6]{BrHer93}. Let $\fm \subset R$ be a graded maximal ideal. Since normality is a local condition, we may replace $R$ with the localization $S^{-1}R$, where $S$ is the set of all homogeneous elements of $R$ not contained in $\fm$, and $S^{-1}\fm$ is its unique graded maximal ideal. Thus, assume that $R$ is a $\mathbb{Z}$-graded Noetherian domain with a unique graded maximal ideal $\fm$. Note that $R_{\fp^*}$ is the localization of $R_{\fm}$ which is an integrally closed domain by assumption. To prove that $R$ is normal, it is sufficient to check Serre's $(S_2)$ and $(R_1)$. 

First, we prove that $R$ is an integral domain. Since associated prime ideals of $R$ are graded by \cite[Lemma 1.5.6]{BrHer93}, they are contained in $\fm$. This shows that the natural map $R \to R_\fm$ is injective. Thus, $R$ is a domain. Next, we check $(S_2)$. Take a prime ideal $\fp$ of $R$ such that $\Ht \fp \ge 2$. If $\fp$ is a graded ideal, then $R_{\fp}$ is the localization of $R_{\fm}$. Hence $R_{\fp}$ is normal. Next suppose that $\fp$ is not graded. Then we have $\Ht \fp=\Ht \fp^*+1$ and $\depth R_{\fp}=\depth R_{\fp^*}+1$ (see \cite[Theorem 1.5.8 and Theorem 1.5.9]{BrHer93}, respectively). Since $R_{\fp^*}$ is a domain with $\dim R_{\fp^*} \ge 1$, it follows that
$$
\depth R_{\fp}=\depth R_{\fp^*}+1 \ge 1+1=2,
$$
hence $(S_2)$ is satisfied, as required.

So it remains to check $(R_1)$. Take a prime ideal $\fp$ of $R$ with $\Ht \fp=1$. If $\fp$ is graded, then $R_{\fp}$ is the localization of $R_{\fm}$ and so it is a discrete valuation domain. So suppose that $\fp$ is not graded. Since $\Ht \fp=1$ and $\Ht \fp=\Ht \fp^*+1$ and $R$ is a domain, we find that $\fp^*=(0)$. Hence $\fp$ does not contain nonzero homogeneous elements. In other words, every nonzero homogeneous element of $R$ is contained in $R \setminus \fp$. Let $R_{(0)}$ be the homogeneous localization of $R$ with respect to $R \setminus \{0\}$. Then we have $R_{(0)} \subset R_{\fp}$. Now \cite[Lemma 1.5.7]{BrHer93} shows that $R_{(0)}=K[t,t^{-1}]$ for some homogeneous element $t \in R$ of positive degree. Let $\mathcal{K}$ be the field of fractions of $R$. Then the field of fractions of $K[t]$ is also $\mathcal{K}$. Hence we get
$$
K[t] \subset R_{\fp} \subset \mathcal{K}.
$$
Since $K[t]$ is a principal ideal domain, it follows from Lemma \ref{PID} that $R_{\fp}$ is the localization of $K[t]$. Hence $R_{\fp}$ is an integrally closed domain of dimension one. This completes the proof of the proposition.
\end{proof}

\begin{remark}
The above proof shows that if $\fp$ is a prime ideal of height one which is not graded in a graded domain $R=\bigoplus_{n \in \mathbb{Z}}R_n$, then the localization $R_{\fp}$ is automatically regular of dimension one.
\end{remark}

Let us recall Flenner's local Bertini theorem \cite[(4.2) Korollar]{Fl77}.

\begin{theorem}[Flenner]
\label{Fl1}
Let $(R,\fm,K)$ be an excellent local ring such that $R$ is an $K$-algebra, where $K$ is an algebraically closed field. Let $x_1,\ldots,x_n \in R$ be a set of generators of the maximal ideal $\fm$ and let $U$ be a non-empty Zariski open subset of $U(x_1,\dots,x_n) \subset \Spec R$. Assume that $R_{\fp}$ has Serre's conditions $(R_{n_1})$ and $(S_{n_2})$ for $\fp \in U$ and a fixed pair of non-negative integers $(n_1,n_2)$. Let us put
$$
\mathbf{x}_{\alpha}:=\sum_{i=1}^n \alpha_i x_i
$$
for $\alpha=(\alpha_1,\ldots,\alpha_n) \in K^n$. Then there exists a Zariski-dense open subset $\mathcal{U} \subset K^n$ such that for every $\alpha \in \mathcal{U}$, the quotient $R_{\fp}/\mathbf{x}_{\alpha}R_{\fp}$ has also $(R_{n_1})$ and $(S_{n_2})$ for all $\fp \in U$.
\end{theorem}

To apply the result from the paper \cite{CGM86}, we need the following.

\begin{lemma}
\label{Zariskiopen}
Let $\mathbb{P}$ be a local property of locally Noetherian schemes such that the following condition holds:
\begin{enumerate}
\item[$\bullet$]
For a scheme map of finite type $f:X \to S$, let us set $$\mathbb{P}(f):=\{x \in X~|~X_{f(x)}~\mbox{is geometrically}~\mathbb{P}~\mbox{at}~x\}.
$$
Then $\mathbb{P}(f)$ is constructible in $X$.
\end{enumerate}
Let $f:X \to S$ be a scheme map of finite type such that $S$ is an integral Noetherian scheme. If the generic fiber $X_{\eta}$ is geometrically $\mathbb{P}$, then there exists an open neighborhood $\eta \in U \subset S$ such that the fibers $X_s$ are geometrically $\mathbb{P}$ for all $s \in U$.
\end{lemma}

\begin{proof}
Set $Y:=X \setminus \mathbb{P}(f)$. Then we have by assumption that $Y$ is also constructible in $X$ and disjoint from $X_\eta$. By Chevalley's theorem \cite[Exercise II 3.19]{Har77}, it follows that $f(Y)$ is constructible in $S$ and $\eta \notin f(Y)$. From this, we see that $f(Y)$ is contained in a proper closed subset of $S$ and therefore, the Zariski closure $\overline{f(Y)}$ does not contain $\eta$. To finish the proof, it suffices to set $U:=S \setminus \overline{f(Y)}$.
\end{proof}

We consider the following conditions borrowed from \cite{CGM86}.

\begin{enumerate}
\item[$(A 1)$]
If $f:X \to S$ is a flat morphism of locally Noetherian schemes such that all of fibers of $f$ are regular and $S$ is $\mathbb{P}$, then $X$ is $\mathbb{P}$.

\item[$(A 2)$]
Assume that $f:X \to S$ is a morphism of finite type such that $S$ is an integral Noetherian scheme and $X$ is excellent. If the generic fiber $X_\eta$ is geometrically $\mathbb{P}$, then there is an open subset $\eta \in U \subset S$ such that the fiber $X_s$ is geometrically $\mathbb{P}$ for every $s \in U$.
\end{enumerate}

We are now ready to prove the following theorem.

\begin{proof}
Proof of Main Theorem \ref{maintheoremA}(1): Note that the localization $R_{\fm}$ is an excellent local ring with residue field $K$ and $\mathbf{x}_{\alpha} \in R$ is a homogeneous element of degree $1$. Take $U:=U(x_1,\ldots,x_n) \subset \Spec R_{\fm}$. Then $U$ is the punctured spectrum of $\Spec R_{\fm}$. Since $R$ is a normal domain by assumption, $R_{\fm}$ is a normal domain, and by Theorem \ref{Fl1} and Serre's normality criterion, we find that $(R_{\fm})_{\fp}/\mathbf{x}_{\alpha}(R_{\fm})_{\fp}$ is normal for all $\fp \in \Spec R_{\fm} \setminus \{\fm\}$ and $\alpha \in \mathcal{U}$, a Zariski open subset of $K^n$. Since the depth of $R_{\fm}$ is at least 3, we have that the depth of $R_{\fm}/\mathbf{x}_{\alpha}R_{\fm}$ is at least 2. Combining everything together, it follows that $R_{\fm}/\mathbf{x}_{\alpha}R_{\fm}$ is a normal domain. Now let us check that $R/\mathbf{x}_{\alpha}R$ is a domain. For this, it is enough to show that the localization map $R/\mathbf{x}_{\alpha}R \to R_{\fm}/\mathbf{x}_{\alpha}R_{\fm}$ is injective.
Assume that we have $\overline{x}=0$ for $x\in R/\mathbf{x}_{\alpha}R$, where $\overline{x}$ is the image of $x$ via the map $R/\mathbf{x}_{\alpha}R \to R_{\fm}/\mathbf{x}_{\alpha}R_{\fm}$. Then there exists an element $s \in (R/\mathbf{x}_{\alpha}R) \setminus \overline{\fm}$ such that $s \cdot x=0$, where $\overline{\fm}$ is the image of $\fm$ in $R/\mathbf{x}_{\alpha}R$. On the other hand, every associated prime of $R$ is graded by \cite[Lemma 1.5.6]{BrHer93}, which implies that
$$
\bigcup_{\fp \in \Ass (R/\mathbf{x}_{\alpha}R)} \fp \subset \overline{\fm}.
$$
Then this implies that $s$ is not a zero divisor. Therefore, we must have $x=0$. Finally after applying Proposition \ref{localization}, we find that $R/\mathbf{x}_{\alpha}R$ is a normal domain. 

Proof of Main Theorem \ref{maintheoremA}(2): Let $x_1,\ldots,x_n \in R_1$ be the homogeneous system of parameters as stated in $(2)$.
Since $K[x_1,\ldots,x_n]$ is a polynomial algebra, we have $\Proj(K[x_1,\ldots,x_n]) \cong \mathbb{P}^{n-1}_{K}$. By assumption, $K[x_1,\ldots,x_n] \to R$ is a module-finite map between graded rings, so we have a finite map\footnote{Here, the module-finite condition ensures that $\phi$ is indeed a morphism. In general, the map of graded rings does not necessarily induce a map on the corresponding projective schemes; rather it is defined over some subset.}
$$
\phi:\Proj(R) \to \mathbb{P}^{n-1}_{K}.
$$
Since the characteristic of $K$ is $0$, the map $\phi$ has separably generated residue field extensions. We want to apply \cite[Theorem 1]{CGM86} in the case that $\mathbb{P}=``\mbox{normal}"$. By \cite[9.9.5]{Gr65}, the condition of Lemma \ref{Zariskiopen} for $\mathbb{P}=``\mbox{normal}"$ is satisfied. This, together with \cite[Theorem 23.9]{M86}, implies that the conditions $(A 1)$ and $(A 2)$ hold true. By applying \cite[Theorem 1]{CGM86}, there is a non-empty open subset $U \subset (\mathbb{P}^{n-1}_{K})^*$ (the dual projective space) such that $\phi^{-1}(H)$ is normal with $H \in U$. In other words, if $\mathbf{x}_{\alpha}=0$ is the defining equation of $H$, the projective scheme $\Proj(R/\mathbf{x}_{\alpha}R)$ is normal. Let us put $B:=R/\mathbf{x}_{\alpha}R$ for simplicity. The rest of the proof follows the line as that of Proposition \ref{localization}. Therefore, it suffices to check the Serre's $(R_1)$ and $(S_2)$ for $B$. To this aim, choose a prime ideal $\fp \subset B$ such that $\bigoplus_{n>0}B_n \not\subset \fp$ and define $\fp^*$ as previously. So we have
\begin{equation}
\label{normalpart}
B_{(\fp)}=B_{(\fp^*)}=(B_{(\fp^*)})_0[t,t^{-1}],
\end{equation}
where $B_{(\fp)}$ and $B_{(\fp^*)}$ are the homogeneous localizations and $(B_{(\fp^*)})_0$ is the degree $0$ part. This isomorphism can be checked by noting the fact that $B_{(\fp)}$ possesses a unit element of degree $1$. Indeed, suppose that $B_1 \subset \fp$. Then we have $B_1 \subset \fp^*$. Since $K[B_1] \to B$ is module-finite, it follows that $K=K[B_1]/K[B_1] \cap \fp^* \to B/\fp^*$ is a module-finite extension of graded domains. Hence $\fp^*$ is the graded maximal ideal of $R$ and $\bigoplus_{n>0} B_n \subset \fp^* \subset \fp$, which is a contradiction. Thus, $B_1 \setminus \fp$ is not empty. By the normality of $\Proj(R/\mathbf{x}_{\alpha}R)$, it follows that $(B_{(\fp^*)})_0$ is normal. So by $(\ref{normalpart})$, $B_{(\fp)}$ is also normal. As the usual localization $B_{\fp}$ is the localization of $B_{(\fp)}$, the normality of $B_{\fp}$ follows. Next, take  a prime ideal $\fp \subset B$ such that $\bigoplus_{n>0}B_n \subset \fp$. Then we have $\fp=\fp^*=\bigoplus_{n>0}B_n$ and $\depth B_{\fp}=\depth B_{\fp^*}=\depth R_{\fm}/\mathbf{x}_\alpha R_{\fm} \ge 2$ by assumption. This finishes the proof of the theorem.
\end{proof}

\begin{remark}
\begin{enumerate}
\item
Note that the proof given in Main Theorem \ref{maintheoremA}(2) also works for the case $(1)$ by considering an embedding $\Proj(R) \hookrightarrow \mathbb{P}^{n-1}_{K}$ coming from the surjection $K[X_1,\ldots,X_n] \twoheadrightarrow R$ by letting $X_i \mapsto x_i$. Even when $K$ has prime characteristic, this embedding has trivial residue field extensions, thus being separable.

\item
The existence of homogeneous system of parameters $x_1,\ldots,x_n \in R_1$ as required in Main Theorem \ref{maintheoremA}(2) is ensured by Graded Noether Normalization Theorem under the assumption that the field $K$ is infinite, which is the case in our setting.  
\end{enumerate}
\end{remark}

As a corollary, we obtain the following Grothendieck-Lefschetz type theorem for divisor class groups on graded normal domains that are not necessarily standard graded.

\begin{proof}[Proof of Main Theorem \ref{ClassGroup}]
By Demazure's theorem, there exists a normal connected projective variety $X$ over $K$ together with an ample $\mathbb{Q}$-Cartier $\mathbb{Q}$-divisor $D$ such that $R \cong \bigoplus_{n \ge 0} H^0(X,\mathcal{O}_X(nD))$. Write
\begin{equation}
\label{Demazure}
D=\sum \frac{p_V}{q_V} V,
\end{equation}
where $V$ ranges over irreducible and reduced Weil divisors on $X$ such that  $p_V,q_V \in \mathbb{Z}$, $q_V>0$ and $(p_V,q_V)=1$. Such a presentation is unique. Then by \cite[Theorem 1.6]{W81}, there is an exact sequence
$$
0 \to \mathbb{Z} \xrightarrow{\theta} \Cl(X) \to \Cl(R) \to F \to 0,
$$ 
where $\theta(1):=\ell D$ with $\ell:=\mbox{LCM}\{q_V~|~V\}$ ($V$ as in $(\ref{Demazure})$) and the group $F$ is determined as follows: Consider the map $\alpha:\mathbb{Z} \to \bigoplus_V \mathbb{Z}/q_V\mathbb{Z}$ defined by $\alpha(1)=(\overline{p_V})_V$, where $\overline{p_V}$ denotes the class of $p_V$ modulo $q_V\mathbb{Z}$. Then the group $F$ is defined as the cokernel of $\alpha$.

By Main Theorem \ref{maintheoremA}, there exists a Zariski-dense open subset $\mathcal{U} \subset K^n$ for which $R/\mathbf{x}_{\alpha}R$ is a graded normal domain with $\alpha \in \mathcal{U}$. Write $f:=\mathbf{x}_{\alpha}$ for simplicity. Denote by $H_f$ the Cartier divisor on $X$ defined by $f$. Since we are working over characteristic zero, we want to apply Bertini's theorem for linear systems. By \cite[Corollary 1]{CGM86}, $V \cap H_f$ is reduced for a general member $f \in R_1$ for all $V$ appearing in $D$. On the other hand, $V \cap H_f$ may be chosen to be irreducible for a general member $f \in R_1$ for all $V$ as in $(\ref{Demazure})$ (for example, see \cite{A51} for a proof). For such $H_f$, we have
\begin{equation}
\label{Demazure2}
D \cap H_f=\sum \frac{p_V}{q_V} (V \cap H_f)
\end{equation}
with $V \cap H_f$ being irreducible and reduced. Putting $R(f):=\bigoplus_{n \ge 0} H^0(H_f,\mathcal{O}_{H_f}(n(D \cap H_f)))$, we get the commutative diagram:
$$
\begin{CD}
0 @>>> \mathbb{Z} @>>> \Cl(X) @>>> \Cl(R) @>>> F @>>> 0 \\
@. @| @VVV @VVV @| \\
0 @>>> \mathbb{Z} @>>> \Cl(H_f) @>>> \Cl(R(f)) @>>> F @>>> 0 \\
\end{CD}
$$
where $\Cl(X) \to \Cl(H_f)$ (resp. $\Cl(R) \to \Cl(R(f))$) is induced by $H_f \hookrightarrow X$ (resp. $R \to R(f)$) and the lower exact sequence is obtained by applying \cite[Theorem 1.6]{W81} and the expression $(\ref{Demazure2})$. Now the map $\Cl(R) \to \Cl(R(f))$ factors as $\Cl(R) \to \Cl(R/fR) \to \Cl(R(f))$. Indeed, the composite map of section rings $R \to R/fR \to R(f)$ is induced by taking the $0$-th cohomology long exact sequence coming from the short exact sequence:
$$
0 \to \mathcal{O}_X(-H_f+nD) \to \mathcal{O}_X(nD) \to \mathcal{O}_{H_f}(n(D \cap H_f)) \to 0.
$$
Then it follows from \cite[Theorem 1]{RaSr06}, together with the assumption $\dim R \ge 4$, that the map $\Cl(X) \to \Cl(H_f)$ is injective. So we have that the map $\Cl(R) \to \Cl(R(f))$ stays injective and the same holds for $\Cl(R) \to \Cl(R/fR)$.
\end{proof}

\section{Demazure's construction of non-standard graded rings}

\begin{lemma}
\label{semistandard}
Suppose that $R=\bigoplus_{n \ge 0} R_n$ is a Noetherian graded ring such that there is a standard graded Noetherian ring $A=\bigoplus_{n \ge 0} A_n$ for which there is a module-finite extension of graded rings $A \to R$. Then one can find an integer $N>0$ such that
$$
R^{(d)}=R_0[R_d]
$$
for all $d \ge N$ and $R^{(d)}:=\bigoplus_{n \ge 0} R_{nd}$.
\end{lemma}

\begin{proof}
Let $m_1,\ldots,m_r \in R$ be the set of generators of $R$ as an $A$-module. One may assume  all $m_i$ to be homogeneous elements. Set
$$
N:=\max \{\deg(m_1),\ldots,\deg(m_r)\}.
$$
Let $n \ge N$ and $k>0$ and pick an element $x \in S_{n+k}$. Then one can write
$$
x=\sum_{i=1}^r a_i m_i~\mbox{and}~a_i \in A,
$$
where $a_i \in A$ is a homogeneous element of degree equal to $n+k-\deg(m_i)$. As $A$ is standard graded, one can write $a_i=\sum_{i=1}^s b_i c_i$ for $b_i \in A_{n+k-\deg(m_i)-1}$ and $c_i \in A_1$ by the standard assumption on the graded ring $A$.
Summarizing everything, we obtain $x \in A_1 R_{n+k-1}$. Applying the same discussion to $R_{n+k-1}$ inductively, it follows that
$$
R_{n+k}=A_k R_n.
$$

By the definition of graded rings, we have $R_n R_k \subset R_{n+k}=A_k R_n \subset R_k R_n$ and hence $R_n R_k=R_{n+k}$ for $n \ge N$ and $k>0$. Using this, it is easy to see that $R^{(d)}=R_0[R_d]$ for $d \ge N$.
\end{proof}

Using this lemma, we get

\begin{proof}[Proof of Main Theorem \ref{Non-StandardDemazure}]
Since $R_0[R_1]$ is evidently standard graded, the graded ring $R$ satisfies the hypothesis of Lemma \ref{semistandard}. Thus, $R$ satisfies ``condition $(\#)$" in \cite[page 206]{GW78}. By \cite[Lemma (5.1.2)]{GW78}, the sheaf $\mathcal{O}_X(1):=\widetilde{R(1)}$ is invertible and ample on $X$. Let $D$ be the corresponding ample Cartier divisor. Set $\fm:=\bigoplus_{n>0} R_n$ and $R(X,D):=\bigoplus_{n \ge 0} H^0(X,\mathcal{O}_X(nD))$. Denote by $H_{\fm}^i(R)$ the $i$-th local cohomology supported at $\fm$. Under the condition $(\#)$, there is an exact sequence:
$$
0 \to H_{\fm}^0(R) \to R \xrightarrow{\theta} R(X,D) \to H_{\fm}^1(R) \to 0
$$
in view of \cite[(5.1.6)]{GW78}. Since $R$ is a normal domain of dimension $\ge 2$, it follows that both $H_{\fm}^0(R)$ and $H_{\fm}^1(R)$ vanish by Serre's normality criterion. So we have an isomorphism $R \cong R(X,D)$ as graded $R$-modules. By construction of the map $\theta$, $R \cong R(X,D)$ is a ring isomorphism.
\end{proof}

\begin{acknowledgement}
The author is grateful to Professor L. Katth\"{a}n and and Professor K. Yanagawa for sharing illuminating discussions.
\end{acknowledgement}

\end{document}